\documentclass[a4paper,11pt]{amsart}

\usepackage{amsthm}
\usepackage[latin1]{inputenc}
\usepackage[T1]{fontenc}
\usepackage{amsmath}
\usepackage{graphicx}
\usepackage{epsf, subfigure, verbatim}
\usepackage{epsfig}
\usepackage{amssymb}
\usepackage{amsfonts}

\usepackage{srcltx}

\theoremstyle{plain} 
\newtheorem{theorem}[equation]{Theorem}
\newtheorem{proposition}[equation]{Proposition}

\newtheorem{corollary}[equation]{Corollary}
\newtheorem{question}[equation]{Question}

\theoremstyle{remark}
\newtheorem{remark}[equation]{Remark}
\theoremstyle{definition}

\numberwithin{equation}{subsection}

\begin{document}

\title{WOLF--KELLER  Theorem FOR NEUMANN EIGENVALUES}
\author{Guillaume Poliquin and Guillaume Roy-Fortin}
\thanks{The first author is supported by NSERC and ISM and the second author is supported by ISM}

\begin{abstract}
The classical Szeg\H{o}--Weinberger inequality states that among bounded planar domains of given area, the first nonzero Neumann eigenvalue is maximized by a disk. Recently,
it was shown in \cite{4} that, for simply connected planar domains of given area, the second nonzero Neumann eigenvalue is maximized in the limit by a sequence of domains degenerating to a disjoint union of two identical disks. We prove that Neumann eigenvalues of planar domains of fixed area are not always maximized by a disjoint union of arbitrary disks. This is an analogue of a result by Wolf and Keller proved earlier for Dirichlet eigenvalues.

\end{abstract}

\maketitle

\section{INTRODUCTION AND MAIN RESULTS}

\subsection{Dirichlet and Neumann eigenvalue problems}
 Let $\Omega \subset \mathbb{R}^N$ be a bounded domain (not necessarily connected) of volume $|\Omega|$.
 Throughout the paper, we assume that $|\Omega|=1$.
 Let  $\Delta:= \displaystyle\sum_{i=1}^N \dfrac{\partial^2}{\partial x_i^2}$ be the Laplace operator.
 The Dirichlet eigenvalue problem,
\begin{equation}\label{D}\tag{Dirichlet}
-\Delta u= \lambda u \mbox{ in } \Omega \mbox{ and } u=0 \mbox{ on } \partial\Omega, \nonumber
\end{equation}
has discrete spectrum (see \cite[p.7]{1})\index,
\begin{equation}
0<\lambda_1\leq\lambda_2\leq...\nearrow\infty. \nonumber \\
\end{equation}

According to a classical result of Faber and Krahn, the first
eigenvalue $\lambda_1$ is minimized by a ball. Furthermore, we have
the Krahn inequality, obtained from the Faber-Krahn result,
stating that $\lambda_2$ is minimized by the disjoint union of two
identical balls (see, for instance, \cite{1}, \cite{10} or \cite{8}).

If $\Omega$ has Lipschitz boundary, then it is well known that the Neumann eigenvalue problem,
\begin{equation}\label{N}\tag{Neumann}
-\Delta u= \mu u \mbox{ in } \Omega \mbox{ and } \frac{\partial u}{\partial n}=0 \mbox{ on } \partial\Omega, \nonumber
\end{equation}
has discrete spectrum (see \cite[p.113]{1}),
\begin{equation}
\mu_0=0\leq \mu_1 \leq \mu_2 \leq...\nearrow\infty. \nonumber
\end{equation}
Here $\frac{\partial}{\partial n}$ denotes the outward normal
derivative. The first eigenvalue of the Neumann problem $\mu_0=0$
corresponds to constant eigenfunctions.

We know, from a classical result of Szeg\H{o}-Weinberger (see
\cite{6} or \cite{7}), that the ball maximizes $\mu_1$ in all
dimensions.

\begin{remark}
For a disconnected domain $\Omega = \Omega_1 \cup \Omega_2$, the spectrum $\sigma(\Omega)$ is the
ordered union of $\sigma(\Omega_1)$ and $\sigma(\Omega_2)$. If $\Omega$ has $n$ connected components, then  the Neumann eigenvalues $\mu_0
= \mu_1 = ... = \mu_{n-1} = 0$.
\end{remark}

\begin{remark}
In some rare cases Neumann eigenvalues can be calculated explicitly (see
\cite{2}). For instance, the rectangle with sides $a$ and $b$ has
eigenvalues
\begin{equation}\label{eqsquares}
\mu_{j,k}(\Omega) = {\pi^2}\Big(\dfrac{j^2}{a^2} +
\dfrac{k^2}{b^2}\Big); \quad j,k \in \mathbb{N} \cup
\{0\},\end{equation} and the disk of unit area has
eigenvalues
\begin{equation}\label{eqdisks}
\mu_{m,n}(\Omega) = \pi j_{m,n}'^2; \quad m \in \mathbb{N} \cup \{0\}, n \in \mathbb{N}.\end{equation}
Here, $J_n$ is the $n$-th  Bessel function of the first kind and
$j_{m,n}'$ is the $m$-th zero of its derivative $J'_n$.

\end{remark}

\subsection{Main results}

The starting point of our research is a theorem by  Wolf and Keller
(see \cite[Theorem 8.1]{5}), stating that the Dirichlet eigenvalues $\lambda_n$
of planar domains are not always minimized by disjoint unions of disks (note
that for $n=1,2,3,4$ it is either known or conjectured that the
minimizers are disks or disjoint unions of disks).

More precisely, Wolf and Keller showed that $\lambda_{13}$ of any disjoint
union of disks is bigger than $\lambda_{13}$ of a single square
(see \cite[p. 408]{5}). Later, Oudet obtained numerical candidates
that were no longer disjoint unions of disks starting with $\lambda_5$ (see
\cite{12}).

In the present paper, we ask

\begin{question}\label{Q1}
Are disjoint unions of disks maximizing $\mu_n$ for all $n$?
\end{question}
Let us emphasize that, as in \cite{5}, we allow disjoint unions of {\it
arbitrary} disks.

In a recent paper of Girouard, Nadirashvili and Polterovich (see
\cite{4}) it is shown that for simply-connected planar domains of
fixed area,  the second positive Neumann eigenvalue $\mu_2$ is
maximized by a family of domains degenerating to the disjoint union
of two disks of equal area.  The same authors also made a remark
(see \cite[Remark 1.2.8]{3}) that a disjoint union of  $n$ identical disks can not
maximize $\mu_n$ for sufficiently large values of $n$.

In the present paper, we give a negative answer to Question \ref{Q1}:

\begin{theorem}\label{thm3}
$\mu_{22}$ is not maximized by any disjoint union of disks.
\end{theorem}

To prove this theorem, we present an adaptation of Wolf-Keller's
result (see \cite[p. 74]{1}) to the Neumann case. We shall use the
same kind of notation as in Wolf-Keller's paper. Let $\mu_n^* =
\sup{\mu_n(\Omega)}$, which is finite (see \cite{11}). Assuming that
the preceding supremum is attained for a certain domain, let
$\Omega_n^*$ be a maximizer of $\mu_n$ among all domains of unit
volume. Also, we denote by $\alpha\Omega$ the image of $\Omega$ by a homothety with  $\alpha$.

\begin{theorem}\label{thm2}
  Let $n\geq2$. Suppose that $\Omega_n^*$ is the disjoint union of less than $n$ domains in
  $\mathbb{R}^N$, each of positive volume,  such that their total volume equals $1$.
   Then  \\
\begin{equation} \label{part A}
(\mu_n^*)^{N/2}=(\mu_i^*)^{N/2} + (\mu_{n-i}^*)^{N/2} =\max_{1\leq j \leq \frac{n}{2}} \left\lbrace  \\
(\mu_j^*)^{N/2} + (\mu_{n-j}^*)^{N/2}\right\rbrace,
\end{equation}
where $i$ is an integer maximizing $(\mu_j^*)^{N/2} + (\mu_{n-j}^*)^{N/2}$ for $j\leq \dfrac{n}{2}$. Moreover, we have
\begin{equation}\label{part B}
\Omega_n^*=\left(\left(\dfrac{\mu_i^*}{\mu_n^*}\right)^{\frac{1}{2}} \:\Omega_i^*\right)\:\bigcup\:\left(\left(\dfrac{\mu_{n-i}^*}{\mu_n^*}\right)^{\frac{1}{2}}\:\Omega_{n-i}^*\right),
\end{equation}
where the union above is disjoint.
\end{theorem}

Thus, if $\mu_i^*$ is known for $1 \leq i \leq n$,  using formula
\eqref{part A} we can find whether there exists a disconnected
domain $\Omega_n^*$ that would achieve $\mu_n^*$. If
$(\mu_i)^{N/2}+(\mu_{n-i})^{N/2}<(\mu_n^*)^{N/2}$ for all $i$, $1
\leq i \leq \dfrac{n}{2}$, then it is clear that $\Omega_n^*$ must
be connected. Furthermore, if we know $\mu_i^*$ for $i \leq i \leq
n-1$, but don't know $\mu_n^*$, we can sometimes show that
$\Omega_n^*$  is connected:

\begin{corollary}\label{cor2}
Let $\Omega$ be a domain such that $|\Omega| = 1$ and
\begin{equation} (\mu_n(\Omega))^{N/2} > (\mu_i^*)^{N/2} + (\mu_{n-i}^*)^{N/2} \end{equation}
for all $i$, $1 \leq i \leq \frac{n}{2}$, then $\Omega_n^*$ must be connected.
\end{corollary}

In order to prove Theorem \ref{thm3}, we use Theorem \ref{thm2} iteratively.  In other words, to find $\mu_n^*$ in a specific class of domains (i.e., disjoint unions of either disks or squares), we use the results obtained already for $\mu_k^*,\, k=1,..,n-1$  in this class and choose the eigenvalue of either the connected domain or of the ``best'' disjoint union given by \eqref{part A}, whichever is bigger. In this way we obtain sharp upper bounds for arbitrary disjoint unions of either disks or squares, which we in turn compare between themselves in order to find a case where biggest eigenvalue yielded by the disks is lower than that of the squares (see section \ref{S2} for details).

\subsection{Discussion}
In dimensions $N\geq3$, the situation is more complicated than in the planar case. Indeed, as discussed in \cite[p. 562]{AB}, the eigenvalues of the ball have not yet been studied systematically for $N\geq4$. In three dimensions, explicit formulas for  eigenvalues of a ball can be obtained in terms of the roots $a'_{p,q}$ of the derivative of the spherical Bessel function $j_p(x)$ (see \cite{ABRA} for details regarding the spherical Bessel functions and refer to \cite{TABLE} for a table of their zeros $a'_{p,q}$). In an attempt to answer the analogue of Question \ref{Q1} in the three--dimensional case, we conducted numerical experiments for $n=1,...,640$. However, for all these $n$, there exists a disjoint union of balls whose corresponding eigenvalue is bigger than that of the cube.

Finally, we remark that among disconnected domains, the second nonzero Neumann eigenvalue is maximized by a disjoint union of two identical balls for all $N\geq3$ by Theorem \ref{thm2}. Also, $\mu_1=\mu_2=...=\mu_N$ for an $N$-dimensional ball, and  therefore a single ball always yields a lower second nonzero eigenvalue than the
disjoint union mentioned above. Taking this into account together with the results of \cite{4}, we may pose the following

\begin{question}
Is the disjoint union of two identical balls maximizing $\mu_2$ in all dimensions?
\end{question}

Going back to the planar case, we conclude the discussion by the following result, whose analogue for the Dirichlet eigenvalues was proved in \cite[p. 399]{5}:
\begin{proposition}\label{int_values_lemma} Consider the first and the second nonzero eigenvalues $\mu_1, \mu_2$ of the Neumann problem and their respective maxima $\mu_1^*=\pi\, j_{1,1}^{'2},\mu_2^*=2\pi\, j_{1,1}^{'2} $ in the class of disjoint unions of simply connected domains of total unit area. Then, for $i=1,2$, there exists a domain $\Omega_t$ in that class such that
$$\mu_i(\Omega_t) = t,$$
for all values of t in the interval $[0, \mu_i^*]$.
\end{proposition}

\section{PROOFS}

\subsection{Maximization of Neumann eigenvalues for disconnected domains}\label{S1}

In this section, we present the proofs of Theorem \ref{thm2} and Corollary \ref{cor2}.

\begin{proof}[Proof of Theorem \ref{thm2}] Our  proof is similar to
the proof of \cite[Theorem 8.1]{5}. Let $\Omega_n^*$ be the disjoint
union of $\Omega_1$ and $\Omega_2$ with $|\Omega_1|>0$,
$|\Omega_2|>0$ and $|\Omega_1|+|\Omega_2|=1$. Let $u_n$ be an
eigenfunction corresponding to the eigenvalue $\mu_n^*$ on
$\Omega_n^*$. Then, $u_n$ is not identically zero on one of the
components of $\Omega_n^*$, say, on $\Omega_1$ and we have $\mu_n^*$=$\mu_i(\Omega_1)$
for some $0\leq i\leq n$.

At the same time, since $\sigma(\Omega_n^*)$ is the ordered union of $\sigma(\Omega_1)$ and $\sigma(\Omega_2)$, we have $\mu_n^*\geq\mu_{n-i}(\Omega_2)$. Assume $\mu_n^*>\mu_{n-i}(\Omega_2)$. Then, since $\mu_n^*= \min\left\lbrace \mu_i , \mu_{n-i} \right\rbrace $, we can increase the value of $\mu_n^*$ by increasing the volume of $\Omega_1$ and by decreasing the volume of $\Omega_2$ while keeping the total volume equal to $1$. This would contradict the definition of a maximizer. Therefore, $\mu_n^*=\mu_i(\Omega_1)=\mu_{n-i}(\Omega_2)$. Note that $i$ can not be either $0$ or $n$ since it would imply that $\mu_n^*=0$.

We now optimize our choice of domains. Replacing $\Omega_1$ by
$|\Omega_1|^{\frac{1}{N}}\Omega_i^*$, we have a domain that has the same volume
as $\Omega_1$ (note that $|\alpha
\Omega|=\alpha^N|\Omega|$). Moreover, since $\Omega_i^*$ maximizes $\mu_i$, we
get a bigger associated eigenvalue. We do the same for $\Omega_2$.
Thus, we get
\begin{equation}\label{preuve1}
 \mu_n^*=\mu_i(\Omega_1)=\mu_i(|\Omega_1|^{\frac{1}{N}}\Omega_i^*)=\frac{1}{|\Omega_1|^{2/N}}\mu_i(\Omega_i^*)=\frac{1}{|\Omega_1|^{2/N}}\mu_i^*,
\end{equation}
and
\begin{equation}\label{preuve2}
 \mu_n^*=\mu_{n-i}(\Omega_2)=\mu_{n-i}(|\Omega_2|^{\frac{1}{N}}\Omega_{n-i}^*)=\frac{1}{|\Omega_2|^{2/N}}\mu_{n-i}(\Omega_{n-i}^*)=\frac{1}{|\Omega_2|^{2/N}}\mu_{n-i}^*.
\end{equation}

Using \eqref{preuve1}, we find that $|\Omega_1|^{2/N} = \dfrac{\mu_i^*}{\mu_n^*}$, and similarly, we get
$|\Omega_2|^{2/N} = \dfrac{\mu_{n-i}^*}{\mu_n^*}$ from \eqref{preuve2}. Since $|\Omega_1|$+$|\Omega_2|=1$, we have $(\mu_n^*)^{N/2} = (\mu_i^*)^{N/2} + (\mu_{n-i}^*)^{N/2}$. This implies the first equality of \eqref{part A} and \eqref{part B}.

Let us now consider  $\widetilde{\Omega}_j$ for $j=1,2,...,n-1$ defined by
\begin{equation}
\widetilde{\Omega}_j=\left[\left(\frac{(\mu_{j}^*)^{N/2}}{(\mu_j^*)^{N/2}+(\mu_{n-j}^*)^{N/2}}\right)^{\frac{1}{N}}\:\Omega_j^*\right]\:\bigcup\:\left[\left(\frac{(\mu_{n-j}^*)^{N/2}}{(\mu_j^*)^{N/2}+(\mu_{n-j}^*)^{N/2}}\right)^{\frac{1}{N}}\:\Omega_{n-j}^*\right]. \nonumber
\end{equation}
Each $\widetilde{\Omega}_j$ has a unit volume. Moreover, we must remark that
\begin{align*}
 \mu_j\left(\left(\dfrac{(\mu_{j}^*)^{N/2}}{(\mu_j^*)^{N/2}+(\mu_{n-j}^*)^{N/2}}\right)^{\frac{1}{N}}\Omega_j^*\right) & = \mu_{n-j}\left(\left(\dfrac{(\mu_{n-j}^*)^{N/2}}{(\mu_j^*)^{N/2}+(\mu_{n-j}^*)^{N/2}}\right)^{\frac{1}{N}}\Omega_{n-j}^*\right) \\
&= \left[ \left( \dfrac{(\mu_j^*)^{N/2} + (\mu_{n-j}^*)^{N/2}}{(\mu_{n-j}^*)^{N/2}}\right)^{\frac{1}{N}}\right]^2 \mu_{n-j}(\Omega_{n-j}^*) \\
&= ((\mu_j^*)^{N/2}+(\mu_{n-j}^*)^{N/2})^{\frac{2}{N}}.
\end{align*}
Since $\mu_n^*\geq\mu_n(\widetilde{\Omega}_j)$ for all $j$, $1 \leq j \leq n-1$ and $\mu_n^*=\mu_n(\widetilde{\Omega}_i)$ for some $i$, $1\leq i\leq n-1$, we have that $\mu_n^*$ is the maximum value of $\mu_n(\widetilde{\Omega}_j)$. Hence $\Omega_n^* = \widetilde{\Omega_j}$ for any index $j$ realizing the maximum, which proves the last equality in formula \eqref{part A}. Clearly, if $j \leq \dfrac{n}{2}$, then $n-j \geq \dfrac{n}{2}$, and therefore we only need to consider values of $j$ less than or equal to $\dfrac{n}{2}$.

\end{proof}

\begin{proof}[Proof of Corollary \ref{cor2}]
 By hypothesis, we have that $(\mu_n(\Omega))^{N/2} > (\mu_i^*)^{N/2} + (\mu_{n-i}^*)^{N/2}$ and by the definition of a maximizer, we know that $\mu_n^*\geq\mu_n(\Omega)$. Then, Theorem \ref{thm2} can not hold for a disconnected domain.
\end{proof}

\subsection{Disks don't always maximize $\mu_n$}\label{S2}

In this section, we give details of computations that led to Theorem \ref{thm3}.

We calculate the first twenty-two nonzero eigenvalues of the Neumann problem for disjoint unions of squares (right hand side of the Table \ref{table2}) and for disjoint unions of disks (left hand side of the Table \ref{table2}). Let $\mu_n(\textbf{D})$ be the eigenvalues obtained from a single disk, $\mu_n^*(\textbf{UD})$ be the eigenvalues obtained from disjoint unions of disks using Theorem \ref{thm2} and $\mu_n^*$ be the maximizer among all disjoint unions of disks.

\begin{table}[b]\label{table2}
\caption{Maximal eigenvalues for disjoint unions of disks and disjoint unions of
squares computed using Theorem \ref{thm2}.} \centering \small
\begin{tabular}
{@{\extracolsep{\fill}} c c c c c c || c c c c}
\hline
1 & 2 & 3 & 4 & 5 & 6 & 7 & 8 & 9 & 10\\
\hline
n & $\mu_n(\textbf{D})$ & $\mu_n(\textbf{D})$ & $\mu_n^*(\textbf{UD})$ & $\mu_n^*$ & $\mu_n^*$ & $(j^2 + k^2)$ & $\mu_n^*(\textbf{US})/\pi²$ & $\mu_n^*$ & $\mu_n^*$ \\
\hline
1  & $\pi j_{1,1}^{'2}$ & 10.650 &  -  & $\mu_1$ & 10.65 & 1+0 & - & $\mu_1$ & 9.87 \\
2  & $\pi j_{1,1}^{'2}$ & 10.650 &  21.300  & $2 \mu_1$ & 21.30 & 0+1 & 2 & $2\mu_1$ & 19.74 \\
3  & $\pi j_{2,1}^{'2}$ & 29.306 & 31.950 & $3 \mu_1$ & 31.95 & 1+1 & 3 & $3\mu_1$ & 29.61  \\
4  & $\pi j_{2,1}^{'2}$ & 29.306 & 42.599 & $4 \mu_1$ & 42.60 & 4+0 & 4 & $4\mu_1=\mu_4$ & 39.48   \\
5  & $\pi j_{0,2}^{'2}$ & 46.125 & 53.249 & $5 \mu_1$ & 53.25 & 0+4 & 5 & $5\mu_1$ & 49.35   \\
6  & $\pi j_{3,1}^{'2}$ & 55.449 & 63.899 & $6 \mu_1$ & 63.90 & 4+1 & 6 & $6\mu_1$ & 59.22   \\
7  & $\pi j_{3,1}^{'2}$ & 55.449 & 74.549 & $7 \mu_1$ & 74.55 & 1+4 & 7 & $7\mu_1$ & 69.09     \\
8  & $\pi j_{4,1}^{'2}$ & 88.833 & 85.199 & $\mu_8$ & 88.83 & 4+4 & 8 & $8\mu_1=\mu_8$ & 78.96    \\
9  & $\pi j_{4,1}^{'2}$ & 88.833 & 99.483 & $\mu_8+\mu_1$ & 99.48 & 9+0 & 9 & $9\mu_1=\mu_9$ & 88.83        \\
10 & $\pi j_{1,2}^{'2}$ & 89.298 & 110.133 & $\mu_8+2 \mu_1$ & 110.13 & 0+9 & 10 & $10\mu_1$ & 98.70    \\
11 & $\pi j_{1,2}^{'2}$ & 89.298 & 120.783 & $\mu_8+3 \mu_1$ & 120.78 & 9+1 & 11 & $11\mu_1$ & 108.57      \\
12 & $\pi j_{5,1}^{'2}$ & 129.308 & 131.432 & $\mu_8+4 \mu_1$ & 131.43 & 1+9 & 12 & $12\mu_1$ & 118.44    \\
13 & $\pi j_{5,1}^{'2}$ & 129.308 & 142.081 & $\mu_8+5 \mu_1$ & 142.08 & 9+4 & 13 & $13\mu_1=\mu_{13}$ & 128.30   \\
14 & $\pi j_{2,2}^{'2}$ & 141.284 & 152.732 & $\mu_8+6 \mu_1$ & 152.73 & 4+9 & 14 & $14\mu_1$ & 138.17   \\
15 & $\pi j_{2,2}^{'2}$ & 141.284 & 163.382 & $\mu_8+7 \mu_1$ & 163.38 & 16+0 & 15 & $\mu_{15}$ & 157.91  \\
16 & $\pi j_{0,3}^{'2}$ & 154.624 & 177.666& $2 \mu_8$ & 177.67 & 0+16 & 16+1=17 & $\mu_{15}+\mu_1$ & 167.78  \\
17 & $\pi j_{6,1}^{'2}$ & 176.774 & 188.316 & $2 \mu_8+\mu_1$ & 188.32 & 16+1 & 18 & $\mu_{15}+ 2\mu_1$ & 177.65   \\
18 & $\pi j_{6,1}^{'2}$ & 176.774 & 198.965 & $2 \mu_8+2 \mu_1$ & 198.97 & 1+16 & 19 & $\mu_{15}+3\mu_1$ & 187.52  \\
19 & $\pi j_{3,2}^{'2}$ & 201.829 & 209.615 & $2 \mu_8+3 \mu_1$ & 209.62 & 9+9 & 20 & $\mu_{15}+4\mu_1$ & 197.39    \\
20 & $\pi j_{3,2}^{'2}$ & 201.829 & 220.265 & $2 \mu_8+4 \mu_1$ & 220.27 & 16+4 & 21 & $\mu_{15}+5\mu_1$ & 207.26  \\
21 & $\pi j_{1,3}^{'2}$ & 228.924 & 230.915 & $2 \mu_8+5 \mu_1$ & 230.92 & 4+16 & 22 & $\mu_{15}+6\mu_1$ & 217.13 \\
22 & $\pi j_{1,3}^{'2}$ & 228.924 & 241.565 & $2 \mu_8+6 \mu_1$ & 241.56 & 16+9 & 23 & $\mu_{22}$ & 246.74 \\
23 & $\pi j_{7,1}^{'2}$ & 231.156 & 252.215 & $2 \mu_8+7 \mu_1$ & 252.21 & 9+16 & 25+1=26 & $\mu_{22}+\mu_1$ & 256.61 \\
24 & $\pi j_{7,1}^{'2}$ & 231.156 & 266.499 & $3 \mu_8$ & 266.50 & 25+0 & 27 & $\mu_{22}+2\mu_1$ & 266.48\\
25 & $\pi j_{4,2}^{'2}$ & 270.689 & 277.148 & $3 \mu_8+ \mu_1$ & 277.15 & 0+25 & 28 & $\mu_{22}+3\mu_1$ & 276.35 \\

\hline
\end{tabular} \\
\end{table}

\noindent Here is the legend for Table \ref{table2}:
\begin{itemize}
 \item First column represents the index $n$ of the eigenvalue $\mu_n$;
 \item Second column is the eigenvalue $\mu_n$ for a single disk computed from formula \eqref{eqdisks};
 \item Third column gives the numerical value of $\mu_n$;
 \item In the fourth column, we use Theorem \ref{thm2} to provide the numerical value of $\mu_n^*$ under the assumption that the maximizing domain is disconnected.
 \item Fifth column represents the maximum of third and fourth columns in terms of $\mu_j$ of a
 disk. This yields the geometry of the maximizing domain;
 \item Sixth column gives the numerical value of $\mu_n^*$ for disjoint unions of disks;
 \item Seventh column gives the values of $(j²+k²)$  computed using formula \eqref{eqsquares} for eigenvalues of a square;
 \item Eighth column gives $\dfrac{\mu_n^*}{\pi^2}$ computed using Theorem \ref{thm2};
 \item Ninth column gives the maximum of seventh and eighth columns in terms of $\mu_j$ of a
 square. This yields the geometry of the maximizing domain;
 \item Last column gives the numerical value for the maximum of $\mu_n$ for disjoint unions of squares, i.e. \\
 $\max\{\mbox{column 7},\mbox{column 8}\} \pi²$.
\end{itemize}

Theorem \ref{thm2} is used iteratively in order to obtain the
maximizer among disks and squares. Column five allows to
recover the  maximizing domain for $\mu_n$ of disks. For
instance, in the case of $\mu_j$ for $2 \leq j \leq 7$, the
maximizer is given by a disjoint union of $j$ disks of the same
area, since any other combination would yield a lower value of
$\mu_j$. For $\mu_8$, the maximizer is a single disk.
Therefore, the maximizer for $\mu_j$ for $9 \leq j \leq 15$ is one
big disk and $j-8$ smaller ones. For $\mu_{16}$, we consider two
disks of same area, and for $\mu_j$ for $17 \leq j \leq 22$, we have
two big disks and $j-16$ smaller disks. As for the squares, for $\mu_{15}$,
 the maximizer is a single square. As a result,  for $16 \leq j
\leq 21$, we have a big square and $j-15$ smaller squares. For
$\mu_{22}$, the eigenvalue of a single square is bigger than the
possible value obtained using any combination of squares.

\begin{proof}[Proof of Theorem \ref{thm3}]
We see that for $n=22$, the value of the maximizer of any disjoint union of
disks is smaller than the corresponding value of a single square,
i.e. $ 241.56 < 246.74 $. From column 5, we can recover the
geometry of the maximizer among the disjoint unions of disks (see figure
\ref{fig:1} on page \pageref{fig:1}).
\end{proof}

\begin{remark}
 We see that $\mu_{22}$ and $\mu_{23}$ of disjoint unions of squares are bigger than any disjoint union of disks, but we have to wait until $\mu_{83}$ to see that happen once again.
\end{remark}

\graphicspath{{/home/gpoliquin/}}

\begin{figure}[htp]
\begin{center}
\includegraphics[scale=0.58]{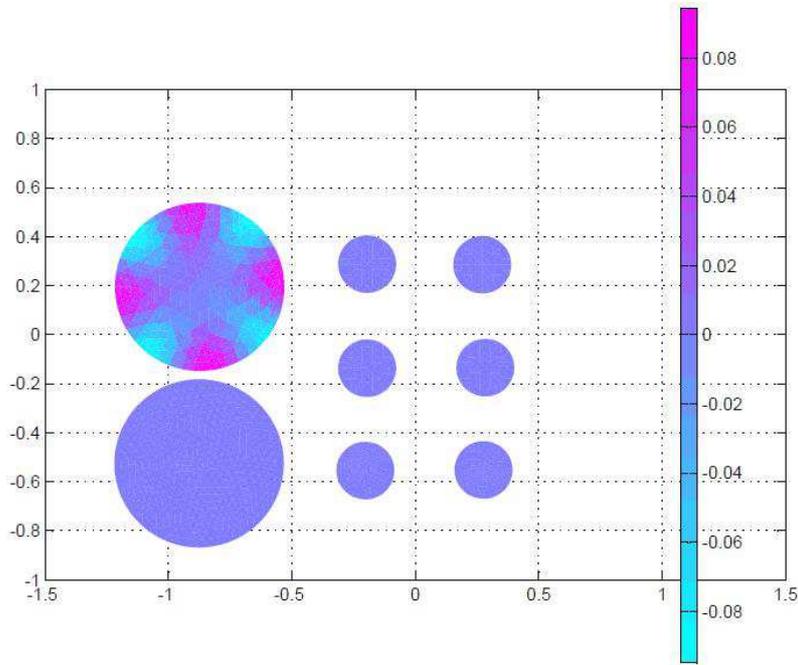}
\caption{A disconnected domain $\Omega$ maximizing $\mu_{22}$ among disjoint unions of disks; $\mu_{22}(\Omega)\approx241.56$}
\label{fig:1}
\end{center}
\end{figure}

Numerical experiments performed in \cite{13} show
that $\Omega_3^*$ is a connected domain different from  a
disk. Taking into account this observation and Theorem \ref{thm3},
as well as the results of \cite{7,6,4} one may reformulate Question
\ref{Q1} in the following way:

\begin{question}
For which $n>3$ is the set $\Omega_n^*$ a disjoint union of disks?
\end{question}

Numerical and analytic results in the Dirichlet case show that the optimal
domain $\Omega_n^*$ is sometimes connected, $n=1,3,5,6,8,9,10$ and
sometimes not, $n=2,4,7$ (see \cite[Figure 5.1]{1} or \cite{12}).
Let us conclude this section by yet another question regarding maximizers for Neumann
eigenvalues:

\begin{question}
 For which $n>3$ are the optimal domains $\Omega_n^*$ connected?
\end{question}

\subsection{Proof of Proposition \ref{int_values_lemma}}

\begin{proof}
For $t=0$ and $i=1$ (resp. $i=2$), it suffices to build any domain with $2$ (resp. $3$) or more connected components. We now consider the case $t>0$.

We begin with the case $i=1$. We let $\Omega_a$ be the ellipse with axes $a/\sqrt{\pi}$ and $1/\sqrt{\pi}a$, for values of $a$ in $(0,1]$. Every such ellipse is bounded and convex in $\mathbb{R}^2$ and, by a result from Kröger (see Theorem 1 in \cite{UppNeuConvDomainEucl}), we have $\mu_m(\Omega_a)d^2_\Omega \leq (2j_{0,1} + (m-1)\pi)^2$, where $d_{\Omega_a}$ is the diameter of the convex domain, namely $2/\sqrt\pi a$ in this case. When applied to the first eigenvalue, the result yields the inequality $\mu_1(\Omega_a) \leq \pi j'_{0,1} a^2$. This means that, for any $t$ in $(0, \pi j_{1,1}^{'2}]$, there exists $\alpha$ in $(0,1]$ such that $\mu_1(\Omega_\alpha) \leq t$. We conclude this part of the proof by invoking a result from Chenais (see \cite[ p. 35 ]{1}), which shows that once $t$ and $\alpha$ are picked as above, the first eigenvalue $\mu_1$ of the family of ellipses $\Omega_a$ varies continuously from $\pi j_{1,1}^{'2}$ to $t$, as $a$ decreases from 1 to $\alpha$.

We now consider the case $i=2$. We know that $\mu_2^*  = 2\pi j_{1,1}^{'2}$. Let $t\in (0, 2\pi j_{1,1}^{'2}]$. Depending on the value of $t$, we construct $\Omega_t$ as a disjoint union of rectangles and circles.

We use formulas \eqref{eqsquares} and \eqref{eqdisks} to compute explicitly the eigenvalues.

We first define $a := \frac{t - \epsilon}{\pi\sqrt{t}}$ and $b := \frac{\pi}{\sqrt{t}}$ and consider the disjoint union of a rectangle with sides $a,b$ and a disk of radius $\sqrt\frac{\epsilon}{t\pi}$. The union has total volume one and it's first nonzero eigenvalue comes from the rectangle for a small enough $\epsilon > 0$. Since $b \geq 1$ for $t \in (0, \pi^2]$, the first nonzero eigenvalue of the rectangle is $\pi^2/b^2 = t$, thus allowing us to cover the desired range.

The idea is similar when $t$ is in the interval $(\pi^2, \pi j_{1,1}^{'2}]$. For $b := \pi/\sqrt t$, we consider the rectangle of sides $b$ and $b - \epsilon$. Letting the other component being the disk whose radius brings the disjoint union to a total area of 1, we choose $\epsilon > 0$ small enough so that the first nonzero eigenvalue comes from the rectangle. Doing so, we get a domain for which $\mu_2 = t$, for all t in $(\pi^2, \pi j_{1,1}^{'2}]$, as desired.

We finally consider the disjoint union of two identical disks of area $1/2$ and scale the first one by a ratio of $j'_{1,1}(2\pi/t)^{1/2}$. Scaling the second one so that the union has unit area allows us to cover the range $t$ in $(\pi j_{1,1}^{'2}, 2\pi j_{1,1}^{'2})$ and concludes the proof.

\end{proof}

\subsection*{Acknowledgments}

The authors are grateful to Iosif Polterovich for suggesting the
problem, as well as for his continuous support throughout this
project. We would like to thank Dan Mangoubi for invaluable help
during the first stages of this research. We are also thankful to
Alexandre Girouard for useful discussions.

\vspace{3 mm}

{\scshape Guillaume Poliquin, Département de mathématiques et de statistique,
Université de Montréal, CP 6128 succ. Centre-Ville, Montréal,
H3C 3J7, Canada}

\emph{E-mail address:}   \verb"gpoliquin@dms.umontreal.ca"

\vspace{3 mm}

{\scshape Guillaume Roy-Fortin, Département de mathématiques et de statistique,
Université de Montréal, CP 6128 succ. Centre-Ville, Montréal,
H3C 3J7, Canada}

\emph{E-mail address:} \verb"groyfortin@dms.umontreal.ca"

\end{document}